\documentclass[11pt,reqno]{amsart}
\usepackage[left=0.9in,top=0.9in,right=0.9in,bottom=0.9in]{geometry}
\usepackage{amsaddr}

\usepackage{amsfonts,amsmath,amsthm,amssymb,latexsym,mathrsfs,stmaryrd}
\usepackage{hyperref}
\usepackage{mathtools}
\usepackage{graphicx}
\usepackage{subcaption}

\usepackage{tikz}\usetikzlibrary{cd,fit,matrix,arrows,decorations.pathmorphing}
\tikzset{commutative diagrams/.cd}

\numberwithin{equation}{section}

\newtheorem{theorem}{Theorem}[section]
\newtheorem{corollary}[theorem]{Corollary}
\newtheorem{lemma}[theorem]{Lemma}
\newtheorem{proposition}[theorem]{Proposition}

\theoremstyle{definition}
\newtheorem{definition}[theorem]{Definition}
\newtheorem{definition-theorem}[theorem]{Definition-Theorem}
\newtheorem{example}[theorem]{Example}

\newtheorem{remark}[theorem]{Remark}

\theoremstyle{remark}

\newtheorem*{remark*}{Remark}

\usepackage{enumitem}
\setenumerate[1]{label=\textup{(\alph*)}}
\setenumerate[2]{label=\textup{(\roman*)}}

\newcommand\Z{{\mathbb Z}}

\newcommand\Q{{\mathbb Q}}
\newcommand\C{{\mathbb C}}
\newcommand\A{{\mathbb A}}

\newcommand\F{{\mathbb F}}
\def\P{{\mathbb P}}

\newcommand\Fq{{{\mathbb F}_q}}

\newcommand\Zp{{{\mathbb Z}_p}}

\def\O{\mathcal O}

\DeclareMathOperator{\im}{im}

\DeclareMathOperator{\Hom}{Hom}
\DeclareMathOperator{\End}{End}
\DeclareMathOperator{\Aut}{Aut}

\DeclareMathOperator{\Spec}{Spec}

\DeclareMathOperator{\GL}{GL}

\DeclareMathOperator{\Mat}{Mat}

\newcommand\subeq{\subseteq}

\newcommand\supeq{\supseteq}

\newcommand\tl{\widetilde} 
\newcommand\hhat{\widehat} 

\newcommand{\map}[1][]{{\xrightarrow{#1}}} 

\DeclarePairedDelimiter{\abs}{\lvert}{\rvert}

\DeclarePairedDelimiter{\set}{\{}{\}}
\DeclarePairedDelimiter{\pairing}{\langle}{\rangle}
\DeclarePairedDelimiter{\parens}{\lparen}{\rparen}
\DeclarePairedDelimiter{\bracks}{\lbrack}{\rbrack}

\newcommand{\Ob}{\mathrm{Ob}}
\newcommand{\m}{\mathfrak{m}}
\newcommand{\zetahat}{\hhat{\zeta}}
\newcommand{\Zhat}{\hhat{Z}}

\newcommand{\FinMod}{\mathbf{FinMod}}
\newcommand{\Dir}{\mathrm{Dir}}
\newcommand{\FinSet}{\mathbf{FinSet}}
\newcommand{\Isom}{\mathrm{Isom}}
\newcommand{\Coh}{\mathrm{Coh}}
\newcommand{\Surj}{\mathrm{Surj}}
\newcommand{\st}{\mathrm{st}}
\newcommand{\Quot}{\mathrm{Quot}}
\DeclareMathOperator*{\Prob}{Prob}
\newcommand{\fa}{\mathfrak{a}}
\newcommand{\Hyp}[5]{{}_{#1}\phi_{#2}\bracks*{
    \begin{matrix}
        #3 \\ #4
    \end{matrix};
    #5}
}

\begin{document}
\title{Commuting matrices via commuting endomorphisms}

\author{Yifeng Huang}
\address{Dept.\ of Mathematics, University of British Columbia}
\email{huangyf@math.ubc.ca}


\begin{abstract}
    Evidences have suggested that counting representations are sometimes tractable even when the corresponding classification problem is almost impossible, or ``wild'' in a precise sense. Such counting problems are directly related to matrix counting problems, many of which are under active research. Using a general framework we formulate for such counting problems, we reduce some counting problems about commuting matrices to problems about endomorphisms on all finite abelian $p$-groups. As an application, we count finite modules on some first examples of nonreduced curves over $\Fq$. We also relate some classical and hard problems regarding commuting triples of matrices to a conjecture of Onn on counting conjugacy classes of the automorphism group of an arbitrary finite abelian $p$-group.
\end{abstract}

\maketitle

\section{Introduction}\label{sec:intro}
Classifying tuples of $n\times n$ matrices over a field $k$ satisfying certain relations, up to simultaneous conjugation by $\GL_n(k)$, is a classical problem in linear algebra. It is equivalent to classifying finite-dimensional representations of a finitely presented associated algebra $A$ over $k$ up to isomorphism. The representation theory of such algebras has been intensively studied (for example, \cite{drozd1980tame,king1994moduli}), through which it has been long known that the full classification problem is often intractable, or ``wild''. Even for commutative algebras, such problems are almost always wild except practically only one nontrivial example, namely, $A=k[x,y]/(xy)$ \cite{drozd1972}. 

However, the problem of counting representations over a finite field sometimes has a surprisingly nice answer despite the extreme difficulty of the corresponding classification problem. For example, Feit and Fine in 1960 determined the number of commuting pairs of matrices in $\Mat_n(\Fq)$ and gave a beautiful generating function \cite{feitfine1960pairs}. For another example, Kac \cite{kac1983root} showed in 1983 that $A_{\mathbf{v}}(q)$, the number of isomorphism classes of absolutely indecomposable representations of a quiver over $\Fq$ of dimension vector $\mathbf{v}$, is a polynomial in $q$, called the Kac's polynomial. He then conjectured that $A_{\mathrm{v}}(q)$ has nonnegative coefficients, which was famously proved in the 2013 Annals paper by Hausel, Letellier, and Rodriguez-Villegas \cite{hlrv2013positivity}. In a sense, these niceness results are not fully expected for general reasons since the corresponding classifications are wild in general.

\begin{remark*}
    The above involves two natural notions to count representations: the na\"ive count where each isomorphism class of representations contributes $1$, and the \textbf{groupoid count} where each isomorphism class $M$ contributes $1/\abs{\Aut(M)}$. Concretely, the former counts matrix tuples up to conjugation, while the latter essentially just counts matrix tuples. However, the former count can be reduced to an instance of the latter count. For example, it follows from the orbit-stabilizer theorem that the number of commuting pairs of matrices in $\Mat_n(\Fq)$ up to conjugation, multiplied by $\abs{\GL_n(\Fq)}$, is the number of commuting triples in $\Mat_n(\Fq)$ where the third matrix is in $\GL_n(\Fq)$.
\end{remark*}

Recent research reveals more examples of such ``curious niceness'' that are far from understood. For example, fix a partition $\lambda$ and let $M_\lambda(p)$ be an abelian $p$-group of type $\lambda$ (see \S \ref{subsec:partition}). Based on explicit evidences, the number of conjugacy classes in $\Aut(M_\lambda(p))$ \cite{onn2008representations} and the number of $\Aut(M_\lambda(p))$-conjugacy classes in $\End(M_\lambda(p))$ \cite{pss2015similarity} are conjectured to be polynomials in $p$ for all $\lambda$. In \cite{pss2015similarity}, nonnegativity of coefficients is also conjectured. But unless $\lambda=(1^n)$, there is no theory of normal forms on $\End(M_\lambda(p))$ in general. More recently, in \cite{huang2023mutually,huangjiang2023torsionfree}, the author and Jiang investigated the groupoid counts of representations over commutative algebras, especially the coordinate rings of reduced singular curves over $\Fq$. All examples known so far display polynomiality, nonnegativity, an analytic well-behavedness with respect to resolution of singularities, as well as a surprising modularity phenomenon that has no direct analogue in previous works. More precisely, say $q$ is an odd prime power and let $R=\Fq[[x,y]]/(y^2-x^h)$ where $h\geq 2$. Define
\begin{equation}
    C_n(R):=\Hom_{\mathbf{AssoAlg}_{\Fq}}(R,\Mat_n(\Fq))
\end{equation}
to be the set of $n$-dimensional representations of $R$; concretely, $C_n(R)$ is the set of pairs of commuting nilpotent matrices $A,B\in \Mat_n(\Fq)$ such that $B^2=A^h$. (By \cite{drozd1972}, classifying finite $R$-modules is a wild problem except when $h=2$.) It is proved in \cite{huangjiang2023torsionfree} that for each $n$, $\abs{C_n(R)}$ is a polynomial in $q$ with nonnegative coefficients. Moreover, the natural generating function 
\begin{equation}\label{eq:zhat-def}
    \Zhat_R(t):=\sum_{n\geq 0} \frac{\abs{C_n(R)}}{\abs{\GL_n(\Fq)}} t^n
\end{equation}
is determined as an explicit power series in $q^{-1}$ and $t$, whose specialization at $t=\pm 1$ gives a modular form.\footnote{If $h$ is odd, the modularity directly results from the Andrews--Gordon--Rogers--Ramanujan identities \cite[Corollary 7.8]{andrewspartitions}. If $h\geq 4$ is even and $t=-1$, the modularity is so far conditional on the truth of a curious Rogers--Ramanujan-type identity \cite{huangjiang2023torsionfree}.} More generally, the author conjectured in \cite{huang2023mutually} that if $R$ is the coordinate ring of an affine reduced singular curve over $\Fq$ and $\tl R$ is its normalization, then $\Zhat_R(t)/\Zhat_{\tl R}(t)$ is entire in $t$. The conjecture, if true, would provide a lot of analytic control for $\Zhat_R(t)$ because $\Zhat_{\tl R}(t)$ is well-known by \cite{cohenlenstra1984heuristics}. In summary, representation counting problems seem to possess hidden structures not seen in the corresponding classification problems. Any new general connections or specific examples in the context of counting representations of (commutative) algebras, if discovered, would be very desirable. 

Remark \ref{rmk} below explains why the motivation of understanding $\Zhat_R(t)$ in a general framework goes beyond the intrinsic interests discussed above. 

\begin{remark}\label{rmk}
    Let $C_{n,m}(k)$ denote the set of $m$-tuples of commuting $n\times n$ matrices over a field $k$. In 1955, Motzkin and Taussky \cite{motzkintaussky1955} proved that $C_{n,2}(\C)$ is irreducible. Gerstenhaber \cite{gerstenhaber1961} raised the question of whether $C_{n,m}(\C)$ is irreducible for general $n,m$, and after deep research \cite{guralnick1992note, guralnicksethuraman2000,holbrookomladic2001,ngosivic2014varieties,sivic2012iii, holbrookomeara2015,jelisiejewsivic}, the cases where $m=3$ and $12\leq n\leq 28$ remain open. This is also closely related to a question posted by Guralnick \cite[p.~74, (i)]{guralnick1992note} about the dimension of $C_{n,m}(\C)$: indeed, $C_{n,m}(\C)$ is reducible if its dimension is strictly greater than $n^2+(m-1)n$, and this is how he proved the reducibility for $m,n\geq 4$ and $m=3,n\geq 32$. Since dimension of $C_{n,3}(\C)$ is essentially the asymptotics of $\abs{C_{n,3}(\Fq)}$ as $q\to \infty$, a specific instance of $\Zhat_R(t)$ (namely, $R=\Fq[x,y,z]$) already encodes strong information towards these geometric problems. Of course, approaching these problems through investigating $\Zhat_{\Fq[x,y,z]}(t)$ (even just asymptotically) is not promising so far, as the latter is probably more difficult (in fact, equivalent to a hard classical problem of counting matrix pairs up to conjugation), unless more general theories of $\Zhat_R(t)$ are developed in the future.
\end{remark}

\subsection{Contents of the paper}
To further the investigation of counting representations over commutative algebras, we formulate a general framework that unifies several current techniques and clarifies the connection between representations and commuting matrices. An important feature is that our framework also works in ``arithmetic'' settings, i.e., when the algebra does not contain a field. The key difference is that working over a commutative $\Z$-algebra, finite representations need to be parametrized by commuting endomorphisms of all possible finite abelian $p$-groups, instead of commuting matrices over a fixed ring. 

As applications, we give the groupoid count of finite modules of given cardinality over the polynomial ring $\Z[T]$ and the nonreduced rings $\Fq[x,y]/(x^b)$ and $\Fq[x,y]/(x^b y)$, $b\geq 2$. Generating functions, sometimes together with exact and asymptotic formulas, are found. The example of $\Z[T]$ is the first ``arithmetic'' example where the groupoid count is explicitly computed after Cohen--Lenstra's result for Dedekind domains \cite{cohenlenstra1984heuristics}. Our formula is what one would expect from Feit--Fine's formula concerning $\Fq[x,y]$ and the analogy between $\Z$ and $\Fq[x]$. However, the case of $\Z[T]$ importantly lacks the commuting matrix pair interpretation as in $\Fq[x,y]$, so our formula is not implied by Feit--Fine's formula. In fact, our simple proof recovers Feit--Fine's formula since it also applies if $\Z$ is replaced by any Dedekind domain, including $\Fq[x]$.  

On the other hand, $R=\Fq[x,y]/(x^b)$ and $R=\Fq[x,y]/(x^b y)$ are the first examples of nonreduced curves whose associated $\Zhat_R(t)$ (see \eqref{eq:zhat-def}) is computed. The proof idea can be simply summarized by: instead of viewing $R$ as an $\Fq$-algebra, view it as an $\Fq[x]$-algebra and use the commuting endomorphism interpretation in place of the commuting matrix interpretation. The same proof actually applies to $b=1$ as well, recovering and generalizing the result of \cite{huang2023mutually} concerning the first singular example of $R$. Remark \ref{rmk:resolve} explains why our new formulas provide an indirect evidence to an analytic conjecture in \cite{huang2023mutually}.

It follows naturally from our framework that the question of finding $\abs{C_{n,3}(\Fq)}$ (see Remark \ref{rmk}) can be reduced to the problem of \cite{onn2008representations} of counting conjugacy classes in the automorphism group of any finite DVR modules. See \S \ref{subsec:triples}, where we briefly discuss the prospect of this point of view.

In addition, as is not mentioned previously in the introduction, we address how the ``framing technique'' applies to our general setting in \S \ref{subsec:framing}. This technique connects our module counting problem to a submodule counting problem via a limiting process, which has played a crucial role in the computation concerning $y^2-x^h=0$ in \cite{huangjiang2023torsionfree}. We also obtain related effective estimates. 

The paper is organized as follows. In Section \ref{sec:prelim}, we give some combinatorial preliminaries mainly used in computing examples. In Section \ref{sec:general}, we describe our general framework and give examples to explain how it corresponds to well-known techniques in combinatorial species theory and counting problems involving modules over associative algebras, quiver representations, finite coherent sheaves over schemes, and permutations. In Section \ref{sec:applications}, we perform the explicit computations.

We conclude the introduction with our formula for the nonreduced nodal curve $x^b y=0$:

\begin{theorem}[{Corollary \ref{cor:nonred-node}}]
    
    For $b\geq 1$, using the notation \eqref{eq:qpoch-notn}, we have
    \begin{align}
        &\sum_{n\geq 0} \frac{\#\set{(A,B)\in \Mat_n(\Fq):AB=BA, A^b B=0}}{\abs{\GL_n(\Fq)}} t^n\\
        & = \frac{1}{(t;q^{-1})_\infty \prod_{i=1}^b (t^i;q^{-1})_\infty}\sum_{k=0}^\infty \frac{q^{-k^2} t^{(b+1)k}}{(q^{-1};q^{-1})_k}  (tq^{-(k+1)};q^{-1})_\infty.
    \end{align}
    
\end{theorem}

\section{Preliminaries}\label{sec:prelim}
In this section, we collect some standard definitions and facts used later.

\subsection{Some $q$-series}
For $n\geq 0$, we use the $q$-Pochhammer symbol
\begin{equation}\label{eq:qpoch-notn}
    (a;q)_n=(1-a)(1-aq)\dots (1-aq^{n-1}), \quad (a;q)_\infty=(1-a)(1-aq)(1-aq^2)\dots.
\end{equation}

We have Euler's identity \cite[Eq.~(2.2.5)]{andrewspartitions}
\begin{equation}\label{eq:euler-identity}
    \sum_{n=0}^\infty \frac{t^n}{(q;q)_n}=\frac{1}{(t;q)_\infty}.
\end{equation}

\subsection{Formal Dirichlet series and filtered algebras}\label{subsec:dirichlet}
In this paper, we will treat Dirichlet series with focus on their coefficients rather than convergence. More precisely, we work in the ring
\begin{equation}
    \Dir_{\C}:=\set{f(s)=\sum_{n=1}^\infty a_n(f) n^{-s}: a_n(f)\in \C},
\end{equation}
called the ring of formal Dirichlet series over $\C$. Addition and multiplication are performed in straightforward manners. For $f\in \Dir_\C$, we use $a_n(f)$ to denote the $n^{-s}$ coefficient (called the $n$-th Dirichlet coefficient) of $f$. For $m\in \Z_{\geq 1}$ and $n\in \Z$, the substitution $f(ms+n)$ is well-defined, giving a ring homomorphism $\Dir_{\C}\to \Dir_{\C}$ sending $f(s)$ to $f(ms+n)$. 

A \textbf{filtration} on a commutative $\Q$-algebra $A$ is a flag of $\Q$-subspaces $F: A=F_0 A\supeq F_1 A \supeq \dots$ such that $F_mA \cdot F_n A \subeq F_{m+n}A$. A filtration induces a structure of topological ring on $A$, in which $\set{F_nA}$ is a system of neighborhoods of $0$. A filtered $\Q$-algebra is a commutative $\Q$-algebra equipped with a filtration. We say a filtered $\Q$-algebra is complete if it is complete with respect to the topology induced by its filtration. 

Like the power series ring, the ring $\Dir_\C$ is a complete filtered $\Q$-algebra with filtration $F_n(\Dir_\C)=\set{f: a_1(f)=\dots=a_n(f)=0}$.

\subsection{Dedekind zeta function}\label{subsec:dedekind}
Let $S$ be a Dedekind domain such that the residue field of every maximal ideal of $S$ is finite and there are only finitely many maximal ideals of $S$ with bounded index. We say in this case that the Dedekind zeta function of $S$ is defined. It is the case if $S=\Z, \Zp, \Fq[T],\Fq[[T]]$, or the ring of integers of any global or local number or function field. 

The Dedekind zeta function of $R$ is
\begin{equation}
    \zeta_R(s)=\sum_{I\subeq R} \abs{R/I}^{-s}\in \Dir_\C,
\end{equation}
summing over all finite-index ideals of $R$. We have the Euler product formula
\begin{equation}\label{eq:dedekind-euler-prof}
    \zeta_R(s)=\sum_{\m} \frac{1}{1-\abs{S/\m}^{-s}},
\end{equation}
where $\m$ ranges over all maximal ideals of $S$. 

\subsection{Partitions and DVR-modules}\label{subsec:partition}
As usual, a partition is a sequence of nonincreasing integers $\lambda=(\lambda_1, \lambda_2,\dots )$ that are eventually zero; the numbers $\lambda_i$ that are positive are called parts of $\lambda$. We identify a partition with its Ferrer--Young diagram, defined as set $\set{(i,j)\in \Z^2: 1\leq j\leq \lambda_i}$; the element $(i,j)$ of the set is called the cell at $i$-th row and $j$-th column. The transpose of a partition $\lambda$, denoted by $\lambda'$, is the partition associated to the transpose Ferrer--Young diagram obtained by switching rows and columns. The $i$-th part of $\lambda'$ is denoted by $\lambda_i'$. For $i\geq 1$, let $m_i(\lambda)$ denote the number of times $i$ appears as a part in $\lambda$. The size of $\lambda$ is $\abs{\lambda}:= \sum_{i\geq 1} \lambda_i=\sum_{i\geq 1} im_i(\lambda)$, and the length of $\lambda$ is $\ell(\lambda):=\lambda_1'=\sum_{i\geq 1} m_i(\lambda)$. 

Let $S$ be a discrete valuation ring (DVR) with maximal ideal $\pi S$ and residue field $\Fq$. Then any finite-cardinality module $N$ over $S$ is isomorphic to a unique module of the form
\begin{equation}\label{eq:needed3}
    N\simeq_S \frac{S}{\pi^{\lambda_1}} \oplus \dots \oplus \frac{S}{\pi^{\lambda_l}}, \quad l\geq 0, \lambda_1\geq \dots \geq \lambda_l>0.
\end{equation}
We say the partition $\lambda=(\lambda_1,\lambda_2,\dots,\lambda_l)$ is the \textbf{type} of $N$. We have $\abs{S/\pi^a}=q^a$ and $\abs{N}=q^{\abs{\lambda}}$.

We will need the $S$-module structure of $\End_S(N)$. Given $a,b\geq 1$, we have
\begin{equation}
    \Hom_S(S/\pi^a,S/\pi^b)=\begin{cases}
        S/\pi^b S, &a\geq b;\\
        \pi^{b-a}S/\pi^b S, &a\leq b.
    \end{cases}
\end{equation}
In particular, $\Hom_S(S/\pi^a,S/\pi^b)\simeq S/\pi^{\min\set{a,b}}$. If $N$ is of the form in \eqref{eq:needed3}, then $\End_S(N)=\bigoplus_{i,j\geq 1} \Hom_S(S/\pi^{\lambda_i},S/\pi^{\lambda_j})$. As a result, the type of $N$ is the partition
\begin{equation}
    1\cdot [\lambda_1]+ 3\cdot [\lambda_2] + 5\cdot [\lambda_3] + \dots + (2i-1)\cdot [\lambda_i] + \dots
\end{equation}
consisting of $(2i-1)$ copies of the part $\lambda_i$ for all $i\geq 1$. We denote this partition by $\lambda^2$. Note that the parts of the transpose of $\lambda^2$ are $(\lambda^2)_i'=\lambda_i'^2$.

\begin{lemma}\label{lem:aut-and-end}
    Let $N$ be a module of type $\lambda$ over a DVR $(S,\pi,\Fq)$. Then
    \begin{enumerate}
        \item $\abs{\Aut_S(N)} = q^{\sum_{i\geq 1}\lambda_i'^2}\prod_{i\geq 1} (q^{-1};q^{-1})_{m_i(\lambda)}$. 
        \item $\abs{\End_S(N)} =  q^{\sum_{i\geq 1}\lambda_i'^2}$.
    \end{enumerate}
\end{lemma}
\begin{proof}
    ~\begin{enumerate}
        \item See for instance \cite[p.~181]{macdonaldsymmetric}.
        \item This follows from the description of $\End_S(N)$ above. \qedhere
    \end{enumerate}
\end{proof}

\subsection{Durfee squares}\label{subsec:durfee}
Given a partition $\lambda$, we arrange its Ferrer--Young diagram such that the $(1,1)$-cell is at the topleft. The (first) Durfee square of $\lambda$ is the largest square that fits in the topleft corner of $\lambda$. The Durfee square divides $\lambda$ into three parts: the square itself, the subpartition to its right, and the subpartition below it. For $i\geq 2$, the $i$-th Durfee square $\lambda$ is recursively defined as the Durfee square of the subpartition below the $(i-1)$-st Durfee of $\lambda$. 

For $i\geq 1$, we denote by $\sigma_i(\lambda)$ the sidelength of the $i$-th Durfee square of $\lambda$, and we define the \textbf{Durfee partition} of $\lambda$ to be $\sigma(\lambda)=(\sigma_1(\lambda),\sigma_2(\lambda),\dots)$. Note that $\abs{\sigma(\lambda)}=\ell(\lambda)$.

We will need the following identities in arguments involving Durfee squares:
\begin{equation}\label{eq:identity-durfee}
    \sum_{\lambda:\ell(\lambda)\leq k} q^{\abs{\lambda}}=\frac{1}{(q;q)_k},\quad \sum_{\lambda:\lambda_1\leq k} q^{\abs{\lambda}} t^{\ell({\lambda})} = \frac{1}{(tq;q)_\infty}, \quad \sum_{\lambda: \lambda_1\leq k, \ell(\lambda)\leq l} q^{\lambda}=\frac{(q;q)_{k+l}}{(q;q)_k(q;q)_l}.
\end{equation}

\section{The general framework}\label{sec:general}

\subsection{Absolute Cohen--Lenstra series}\label{sec:absolute-cl}
Let $\mathcal{R}$ be a category such that every object has a finite automorphism group, and the class $\Ob(\mathcal{R})/{\sim}$ of objects up to isomorphism (called the \textbf{skeleton} of $\mathcal{R}$) is an at most countable set. Let $A$ be a commutative $\Q$-algebra filtered by $F_nA$ that is complete (see \S \ref{subsec:dirichlet}). Let $\mu: \Ob(\mathcal{R})/{\sim}\to A$ be any function such that $\mu^{-1}(A\setminus F_nA)$ is finite for all $n\geq 0$. We call $\mu$ a \textbf{measure} or statistics on $\mathcal{R}$ with values in $A$. Define the \textbf{absolute Cohen--Lenstra series} of $\mathcal{R}$ with respect to $\mu$ by
\begin{equation}
    \zetahat_{\mathcal{R},\mu}:=\sum_{M\in \Ob(\mathcal{R})/{\sim}} \frac{1}{\abs{\Aut_{\mathcal R}(M)}} \mu(M) \in A,
\end{equation}
noting that our assumption ensures that the countable sum converges in $A$. 

\begin{example}\label{eg:absolute-cl}~
    \begin{enumerate}
        \item \label{item:absolute-cl-zeta}
        If $R$ is a commutative ring that is finitely generated over $\Z$, $\mathcal{R}=\FinMod_R$ is the category of finite-cardinality modules over $R$, $A=\Dir_{\C}$ is the ring of formal Dirichlet series over $\C$ (see \S \ref{subsec:dirichlet}), and $\mu(M)=\abs{M}^{-s}$ for a finite $R$-module $M$, then
        \begin{equation}
            \zetahat_{\mathcal{R},\mu}=\sum_{M\in \Ob(\FinMod_R)/{\sim}} \frac{1}{\abs{\Aut_R(M)}}\abs{M}^{-s}=:\zetahat_R(s)
        \end{equation}
        is the series classically considered by Cohen and Lenstra in \cite{cohenlenstra1984heuristics}.
        \item If $Q=(Q_0,Q_1)$ is a (finite) quiver with vertex set $Q_0=\set{x_1,\dots,x_v}$ and arrow set $Q_1$, $\mathcal{R}$ is the category of finite-dimensional quiver representations of $Q$ over a finite field $\Fq$, $A$ is the power series ring $\C[[t_1,\dots,t_v]]$ with the natural degree filtration, and $\mu(M)=t^{\mathbf{d}(M)}:=t_1^{d_1}\dots t_v^{d_v}$ with $\mathbf{d}(M)=(d_1,\dots,d_v)$ being the dimension vector of the quiver representation $M$, then
        \begin{equation}
            \zetahat_{\mathcal{R},\mu}=\sum_{M\in \Ob{(\mathcal R)}/{\sim}} \frac{1}{\abs{\Aut_Q(M)}} t^{\mathbf{d}(M)},
        \end{equation}
        which encodes the weighted count of quiver representations of any given dimension vector.
        \item If $X$ is a finite-type scheme over a finite field $\Fq$, $\mathcal{R}$ is the category of finite-length coherent sheaves over $X$, $A=\C[[t]]$ is the power series ring, and $\mu(M)=t^{\dim_{\Fq} H^0(X;M)}$, then
        \begin{equation}
            \zetahat_{\mathcal{R},\mu}=\sum_{M\in \Ob{(\mathcal R)}/{\sim}} \frac{1}{\abs{\Aut_X(M)}} t^{\dim_{\Fq} H^0(X;M)},
        \end{equation}
        which is the Cohen--Lenstra series $\Zhat_{X/\Fq}(t)$ considered in \cite{huang2023mutually} that has a commuting variety interpretation whenever $X$ is affine. 
        \item \label{item:absolute-cl-gset} If $G$ is a finitely presented group, $\mathcal{R}=\FinSet_G$ is the category of finite $G$-sets, $A=\C[[t]]$, and $\mu(X)=t^{\abs{X}}$ where $\abs{X}$ is the cardinality of the underlying set of a $G$-set $X$, then
        \begin{equation}
            \zetahat_{\mathcal{R},\mu}=\sum_{X\in \Ob(\mathcal{R})/{\sim}} \frac{1}{\abs{\Aut_G(X)}} t^{\abs{X}}.
        \end{equation}
    \end{enumerate}
\end{example}

A ubiquitous theme in such counting problem is the multiplication principle, which in various contexts also appears as the exponentiation formula or the Euler product. We assume $\mathcal{R}$ is a category that admits arbitrary finite direct sums, so $\mathcal{R}$ has an initial object. A \textbf{strictly full subcategory} $\mathcal{R'}$ of $\mathcal{R}$ is a full subcategory such that $M\in \Ob(\mathcal{R}')$ and $M\sim_\mathcal{R} M'$ imply $M'\in \Ob(\mathcal{R}')$.

\begin{definition}\label{def:weak-decomp}
    We say a category $\mathcal{R}$ is the \textbf{weak direct sum} of an at most countable collection of strictly full subcategories $\set{\mathcal{R}_i}_{i\in I}$, denoted by $\mathcal{R}=\bigoplus_{i\in I} \mathcal{R}_i$, if
    \begin{enumerate}
        \item Each $\mathcal{R}_i$ contains the initial object of $\mathcal{R}$.
        \item Every object $M$ in $\mathcal{R}$ is isomorphic to a unique finite coproduct of the form $\coprod_{j=1}^l M_{i_j}$, where $l\geq 0$ and $M_{i_j}\in \Ob(\mathcal{R}_{i_j})$. 
        \item \label{item:aut} For every $M\in \Ob(\mathcal{R})$, in the decomposition above, the natural homomorphism
        \begin{equation}
            \Aut_{\mathcal R}(M_{i_1})\times \dots \times \Aut_{\mathcal R}(M_{i_l})\to \Aut_{\mathcal R}(M)
        \end{equation}
        is a bijection.
    \end{enumerate}
    Given a weak direct sum decomposition, we say a measure $\mu:\Ob(\mathcal{R})/{\sim}\to A$ is \textbf{multiplicative} if $\mu(M)=\mu(M_{i_1})\dots \mu(M_{i_l})$ in the notation above for every $M\in \Ob(\mathcal{R})$. 
\end{definition}

The following is clear from the definition.

\begin{lemma}\label{lem:euler-prod}
    Given a weak direct sum decomposition $\mathcal{R}=\bigoplus_{i\in I} \mathcal{R}_i$ and a multiplicative measure $\mu$ with respect to it, we have
    \begin{equation}
        \zetahat_{\mathcal{R},\mu}=\prod_{i\in I} \zetahat_{\mathcal{R}_i,\mu}.
    \end{equation}
\end{lemma}

\begin{example}\label{eg:euler-prod}~
    \begin{enumerate}
        \item If $\mathcal{R}=\FinMod_\Z$ is the category of finite abelian groups, and $\mathcal{R}_p=\FinMod_{\Zp}$ is the category of finite abelian $p$-groups for each prime $p$, then we have a weak direct sum decomposition $\mathcal{R}=\bigoplus_p \mathcal{R}_p$. Note that the condition \ref{item:aut} of Definition \ref{def:weak-decomp} is met because there are no nontrivial homomorphisms between groups with coprime orders. By taking $\mu(M)=\abs{M}^{-s}$ in the ring $\Dir_\C$ of formal Dirichlet series, we get the Euler product for the classical Cohen--Lenstra series $\zetahat_\Z(s)=\prod_p \zetahat_{\Zp}(s)$ as in \cite{cohenlenstra1984heuristics}.
        \item Similarly, if $R$ is any finitely generated commutative ring over $\Z$, then we have $\FinMod_R=\bigoplus_\m \FinMod_{\hhat{R}_\m}$, where $\m$ ranges over all maximal ideals of $R$ and $\hhat{R}_\m$ denotes the completion of $R$ at $\m$. This gives $\zetahat_R(s)=\prod_\m \zetahat_{\hhat{R}_\m}(s)$.
        \item If $\mathcal{R}=\FinMod_{\Zp}$, then every object in $\mathcal{R}$ is isomorphic to a unique direct sum of indecomposables, namely, $\Z/p^i$ for $i\geq 1$. However, if for $i\geq 1$ we define $\mathcal{R}_i=\set{(\Z/p^i)^n: n\geq 0}$, the condition \ref{item:aut} is not verified because $\Aut(\Z/p \oplus \Z/p^2) \neq \Aut(\Z/p) \times \Aut(\Z/p^2)$. 
        \item If $G$ is a finitely generated group and $\mathcal{R}=\FinSet_G$ is the category of finite $G$-sets, let $\set{X_i}_{i\in I}$ be the collection of finite transitive $G$-sets up to isomorphism. (Concretely, the set corresponds to the set of finite-index subgroups of $G$ up to conjugation.) Let $\mathcal{R}_i=\set{n\cdot X_i: n\geq 0}$, where $n\cdot X_i$ denotes the disjoint union of $n$ copies of $X_i$ as a $G$-set. Then Definition \ref{def:weak-decomp}\ref{item:aut} is met, we have $\mathcal{R}=\bigoplus_i R_i$, and when $\mu(X)=t^{\abs{X}}$ Lemma \ref{lem:euler-prod} is read:
        \begin{align}
            \sum_{X\in \Ob(\FinSet)_G/{\sim}} \frac{1}{\abs{\Aut_G(X)}} t^{\abs{X}} &= \prod_i \prod_{n=0}^\infty \frac{1}{\Aut_G(n\cdot X_i)} t^{n\abs{X_i}} \\
            &= \prod_i \prod_{n=0}^\infty \frac{1}{n! \abs{\Aut_G(X_i)}^n} t^{n\abs{X_i}} \\
            &= \prod_i \exp\left(\frac{t^{\abs{X_i}}}{\abs{\Aut_G(X_i)}}\right) = \exp\left(\sum_i \frac{t^{\abs{X_i}}}{\abs{\Aut_G(X_i)}}\right),
        \end{align}
        recovering the well-known exponential formula in combinatorial species. Note that the exponential form relies on a more explicit understanding of each Euler factor, which is not typically available in the ``linear'' categories in the previous examples. 
        \item \label{item:tad-white} Continuing the example above, but we let $G=\Z^r$ and use a more refined measure $\nu(X)=t^{\abs{X}}y^k\in \C[[t,y]]$, where $k$ is the number of $G$-orbits in $X$. Then $\nu$ is still multiplicative respect to the weak direct sum above, so the refined series $\zetahat_{\FinSet_{\Z^r},\nu}$ can be similarly computed. This is precisely the method of \cite{white2013counting} in his simple alternative proof of Bryan--Fulman's formula \cite{bryanfulman1998orbifold} that counts commuting $r$-tuples of permutations (see also Example \ref{eg:commuting}\ref{item:commuting-permutation}). 
    \end{enumerate}
\end{example}

\subsection{Groupoid of representations relative to a forgetful functor}
In this subsection, given a faithful functor (thought of as a forgetful functor), we define a notion that directly generalizes the commuting variety and representation variety (see Example \ref{eg:commuting}\ref{item:commuting-var}). Recall that a \textbf{groupoid} is a category in which all morphisms are isomorphisms. A \textbf{finite groupoid} is a groupoid in which the class of objects up to isomorphism is a finite set, and the automorphism group of each object is finite. The \textbf{cardinality} of a finite groupoid $X$ is defined as
\begin{equation}
    \abs{X}:=\sum_{x\in \Ob(X)/{\sim}} \abs{\Aut(x)}^{-1}\in \Q.
\end{equation}

Let $\mathcal{R}, \mathcal{S}$ be categories that satisfy the assumptions at the beginning of \S \ref{sec:absolute-cl}. Fix a functor $\Phi:\mathcal{R}\to \mathcal{S}$. Then for any object $N$ in $\mathcal{S}$, we define a groupoid
\begin{equation}
    C_{N,\Phi}:=\set{(M,\iota): M\in \Ob(\mathcal{R}), \iota\in \Isom_{\mathcal{S}}(\Phi(M),N)},
\end{equation}
where an isomorphism $(M_1,\iota_1)\map[\sim] (M_2,\iota_2)$ is an isomorphism $\phi: M_1\to M_2$ such that $\iota_1=\iota_2 \circ \Phi(\phi)$. We in addition assume $\Phi$ satisfies the property that $C_{N,\Phi}$ is a finite groupoid for any $N\in \Ob(\mathcal{S})$.

\begin{remark}\label{rmk:faithful}
    In most (if not all) applications, the functor $\Phi$ will be faithful. In this case, $C_{N,\Phi}$ is a groupoid with no nontrivial automorphism: if $(M,\iota)\in C_{N,\Phi}$ and $\phi:(M,\iota)\to (M,\iota)$ is an automorphism, then $\iota=\iota\circ \Phi(\phi)$. Since $\iota$ is an isomorphism, $\Phi(\phi)$ is the identity morphism on $\Phi(M)$, which implies $\phi$ is the identity on $M$ as $\Phi$ is faithful. Hence, the groupoid cardinality of $C_{N,\Phi}$ is the usual cardinality of its skeleton and it is safe to regard $C_{N,\Phi}$ as a usual set.
\end{remark}

Equip $C_{N,\Phi}$ with the natural action by $\Aut_{\mathcal{S}}(N)$, namely, $g\cdot (M,\iota) = (M,g\circ \iota)$ and $g\cdot \phi=\phi$ for $\phi:(M,\iota_1)\to (M,\iota_2)$, where $g\in \Aut_{\mathcal{S}}(N)$. Define the groupoid
\begin{equation}
    \Coh_{N,\Phi}=\set{M\in \Ob(\mathcal{R}): \Phi(M)\sim_S N},
\end{equation}
where an isomorphism $M_1\map[{\sim}] M_2$ is just an isomorphism $M_1\to M_2$ in $\mathcal{R}$. The following connection between $\Coh_{N,\Phi}$ and a quotient groupoid is well-known.

\begin{lemma}
    In the notation above, we have 
    \begin{enumerate}
        \item $\Coh_{N,\Phi}$ is equivalent to the quotient groupoid $C_{N,\Phi}/\Aut_{\mathcal S}(N)$ (see the proof for a definition).
        \item $\abs{\Coh_{N,\Phi}}=\abs{C_{N,\Phi}}/\abs{\Aut_{\mathcal S}(N)}$. 
        \item For any measure $\mu: \Ob(\mathcal{S})/{\sim}\to A$ on $\mathcal{S}$, we have
        \begin{equation}\label{eq:cl-vs-comm}
            \zetahat_{\mathcal{R},\mu\circ \Phi} = \sum_{N\in \Ob(\mathcal{S})/{\sim}} \frac{\abs{C_{N,\Phi}}}{\abs{\Aut_{\mathcal S}(N)}} \mu(N).
        \end{equation}
    \end{enumerate}
\end{lemma}
\begin{proof}~
    \begin{enumerate}
        \item Let $G=\Aut_{\mathcal S}(N)$. Consider the quotient groupoid $C_{N,\Phi}/G$, namely, the groupoid with objects in $C_{N,\Phi}$ but with 
        \begin{equation}
            \Isom_{C_{N,\Phi}/G}((M_1,\iota_1),(M_2,\iota_2)) = \set{(\phi,g): \phi\in \Isom_{\mathcal R}(M_1,M_2), g\in G\text{ such that }\iota_2\circ \Phi(\phi)=g\circ \iota_1}.
        \end{equation}
        We claim the natural forgetful functor $\Psi: C_{N,\Phi}/G\to \Coh_{N,\Phi}$ is an equivalence of category. It is clearly essentially surjective. To check fully faithfulness, consider $(\phi,g)\in \Isom_{C_{N,\Phi}/G}((M_1,\iota_1),(M_2,\iota_2))$, and note that $\Psi((\phi,g))=\phi$. By the commutative square $\iota_2\circ \Phi(\phi)=g\circ \iota_1$ of isomorphisms in $\mathcal{S}$, giving any $\phi$ uniquely determines $g$. This verifies fully faithfulness and thus the claim. 

        \item By (a), $\abs{C_{N,\Phi}/G}=\abs{\Coh_{N,\Phi}}$. It suffices to prove that the expected formula $\abs{C_{N,\Phi}/G}=\abs{C_{N,\Phi}}/\abs{G}$ holds. Consider the natural functor between the skeleton sets $\Pi: \Ob(C_{N,\Phi})/{\sim}\to \Ob(C_{N,\Phi}/G)/{\sim}$. Fix $[(M,\iota)]_{C_{N,\Phi}/G}\in \Ob(C_{N,\Phi}/G)/{\sim}$, the isomorphism class of $(M,\iota)$ in $C_{N,\Phi}/G$, and focus on its preimage simply denoted by $\Pi^{-1}(M,\iota)$. Then for any isomorphism $(\phi,g): (M',\iota')\to (M,\iota)$ in $C_{N,\Phi}/G$, we have an isomorphism $\phi: (M',\iota')\to (M,g^{-1}\iota)$ in $C_{N,\Phi}$. Hence every element in $\Pi^{-1}(M,\iota)$ is (not necessarily unique) of the form $[(M,\iota')]_{C_{N,\Phi}}$ with $\iota'\in G\cdot \iota$.

        Let $G$ act on $\Pi^{-1}(M,\iota)$ by $g\cdot [(M,\iota')]_{C_{N,\Phi}}=[(M,g\circ \iota')]_{C_{N,\Phi}}$; the action is transitive by the previous discussion. By summing over elements of $\Ob(C_{N,\Phi}/G)/{\sim}$, the desired formula reduces to
        \begin{equation}\label{eq:needed}
            \frac{1}{\abs{\Aut_{C_{N,\Phi}/G}(M,\iota)}} = \frac{1}{\abs{G}} \sum_{x\in \Pi^{-1}(M,\iota)} \frac{1}{\abs{\Aut_{C_{N,\Phi}}(x)}}.
        \end{equation}
    
        To find $\abs{\Aut_{C_{N,\Phi}/G}(M,\iota)}$, let $(\phi,g)$ be an automorphism of $(M,\iota)$. We note that $\phi$ can be an arbitrary element of $\Aut_{\mathcal{R}}(M)$, and $g=\iota^{-1}\Phi(\phi)\iota$ is determined by $\phi$. Thus $\abs{\Aut_{C_{N,\Phi}/G}(M,\iota)}=\abs{\Aut_{\mathcal{R}}(M)}$.
    
        To find $\abs{\Aut_{C_{N,\Phi}}(x)}$, where $x=(M,\iota')$, let $\phi\in \Aut_{C_{N,\Phi}}(x)$. Then $\iota'=\iota'\circ \Phi(\phi)$, so $\Phi(\phi)=1_G$. This shows $\abs{\Aut_{C_{N,\Phi}}(x)}=\abs{\ker(\Phi_M)}$, where $\Phi_M:\Aut_{\mathcal{R}}(M)\to G$ is the group homomorphism induced by $\Phi$. 
    
        Since the summand on the right-hand side of \eqref{eq:needed} is constant in $x$, it remains to find the (honest) cardinality of $\Pi^{-1}(M,\iota)$. This is $\abs{G}/\abs{\mathrm{Stab}(M,\iota)}$ by the orbit-stabilizer theorem, where $\mathrm{Stab}(M,\iota)$ consists of $g\in G$ such that there exists $\phi\in \Aut_{\mathcal R}(M)$ such that $\iota=g\circ \iota \circ \Phi(\phi)$. This is equivalent to $g\in \iota\cdot \im(\Phi_M) \cdot \iota^{-1}$, so $\abs{\mathrm{Stab}(M,\iota)}=\abs{\im(\Phi_M)}$.
    
        Now \eqref{eq:needed} reads
        \begin{equation}
            \frac{1}{\abs{\Aut_{\mathcal{R}}(M)}} = \frac{1}{\abs{G}} \cdot \frac{\abs{G}}{\abs{\im(\Phi_M)}} \cdot \frac{1}{\abs{\ker(\Phi_M)}},
        \end{equation}
        which is true. 

        \item This follows immediately from (b). \qedhere
    
    \end{enumerate}
\end{proof}

We note the flexibility of \eqref{eq:cl-vs-comm}: given any measure $\mu$ on $\mathcal{R}$, then one may use any category $\mathcal{S}$ together with a functor $\Phi:\mathcal{R}\to \mathcal{S}$ through which $\mu$ factors to give a ``commuting variety'' interpretation of $\zetahat_{\mathcal{R},\mu}$. Objects that can be recognized as $C_{N,\Phi}$ are ubiquitous in mathematics, and $C_{N,\Phi}$ can be viewed as the set of all possible ways to upgrade a fixed underlying $\mathcal{S}$-object $N$ to an $\mathcal{R}$-object. Such a set is often thought of as a representation variety, and often has a concrete description in commuting endomorphisms, but see also Example \ref{eg:commuting}\ref{item:commuting-scheme} below.

\begin{example}\label{eg:commuting}
    In the following examples, $\Phi$ is faithful, so $C_{N,\Phi}$ is a set by Remark \ref{rmk:faithful}.
    \begin{enumerate}
        \item \label{item:commuting-var} Let $R=\Fq[T_1,\dots,T_m]/(f_1,\dots,f_r)$, $\mathcal{R}=\FinMod_R$, $\mathcal{S}=\FinMod_{\Fq}$, and $\Phi$ is the usual forgetful functor, then objects in $\mathcal{S}$ are just $\Fq^n$ for some $n\geq 0$, and
        \begin{equation}
            C_{\Fq^n,\Phi}=\set{(A_1,\dots,A_m)\in \Mat_n(\Fq)^m: A_iA_j=A_jA_i, f_j(A_1,\dots,A_m)=0}.
        \end{equation}
        This is the set of $\Fq$-points of the usually considered representation variety of $R$, or a commuting variety with relations. We shall denote $C_{n}(R)=C_{n,\Fq}(R):=C_{\Fq^n,\Phi}$. It is also denoted by $(\Spec R)(\Mat_n(\Fq))$ in \cite{hos2023matrix}, called the set of $n\times n$ matrix points on $\Spec R$. Note also that $C_n(R)=\Hom_{\mathbf{AssoAlg}_{\Fq}}(R,\Mat_n(\Fq))$.
        \item \label{item:commuting-over-ring} More generally, if $S$ is any finitely generated commutative ring over $\Z$, 
        $$R=\allowbreak S\{T_1,\dots,T_m\}/(f_1,\dots,f_r)$$ is a finitely presented associated algebra over $S$ with noncommutative generators $T_i$ and relations $f_j$, $\mathcal{R}$ is the category of finite-cardinality left $R$-modules, $\mathcal{S}=\FinMod_S$, and $\Phi$ is the usual forgetful functor, then for any $N\in \FinMod_S$, we have
        \begin{equation}
            C_{N,\Phi}=C_{N,S}(R):=\set{(A_1,\dots,A_m)\in \End_S(N)^m: f_j(A_1,\dots,A_m)=0}.
        \end{equation}
        Furthermore, if $\mu(N)=\abs{N}^{-s}$ for $N\in \Ob(\mathcal{S})/{\sim}$, then \eqref{eq:cl-vs-comm} reads
        \begin{equation}\label{eq:cl-zeta-in-comm}
            \sum_{M\in \Ob(\mathcal R)/{\sim}} \frac{\abs{M}^{-s}}{\abs{\Aut_R(M)}} = \sum_{N\in \Ob(\mathcal S)/{\sim}} \frac{\abs{C_{N,S}(R)}}{\abs{\Aut_S(N)}}\abs{N}^{-s}.
        \end{equation}
        \item \label{item:commuting-quiver} The category $\mathcal{R}$ above can be viewed as the category of $S$-coefficient representations of a $1$-vertex quiver with $m$ self-loops and relations $f_1,\dots,f_r$. One can readily extend the above to any finite quiver with $v$ vertices, if we set $\mathcal{S}=\set{(N_1,\dots,N_v): N_i\in \FinMod_S}$. If $S=k$ is a field, then $C_{N,\Phi}$ can be considered the variety of representations. The natural forgetful functor $\mathcal{R}\to \mathcal{S}$ can be viewed as  ``taking the dimension vector''.
        \item \label{item:commuting-permutation} If $G=\pairing{s_1,\dots,s_m | w_1,\dots,w_r}$ is a finitely presented group with generators $s_i$ and relations $w_j$, $\mathcal{R}=\FinSet_G$, $\mathcal{S}=\FinSet$ is the category of finite sets, and $\Phi$ is the usual forgetful functor, then objects in $\mathcal{S}$ are just $[n]=\set{1,\dots,n}$ for $n\geq 0$, and
        \begin{equation}
            C_{[n],\Phi}:=\set{(g_1,\dots,g_m)\in (\Sigma_n)^m: w_j(g_1,\dots,g_m)=1},
        \end{equation}
        where $\Sigma_n$ is the permutation group. If $\mu(X)=t^{\abs{X}}$ for $X\in \mathcal{S}$, then \eqref{eq:cl-vs-comm} reads
        \begin{equation}
            \sum_{X\in \Ob(\mathcal{R})/{\sim}} \frac{1}{\Aut_G(X)} t^{\abs{X}} = \sum_{n=0}^\infty \abs{C_{[n],\Phi}} \frac{t^n}{n!}.
        \end{equation}
        In particular, if $G=\Z^r$, then $C_{[n],\Phi}$ is the set of $m$-tuples of commuting permutations in $\Sigma_n$; see also Example \ref{eg:euler-prod}\ref{item:tad-white}.
        \item \label{item:commuting-scheme} One can generalize \ref{item:commuting-var} to a scheme-theoretic setting. Let $\pi:X\to Y$ be a morphism between finite-type separated schemes over $\Spec \Z$. Let $\mathcal{R},\mathcal{S}$ be the category of finite-length coherent sheaves on $X$, $Y$, respectively. Let $\Phi=\pi_*$ be the pushforward functor. Then $C_{N,\Phi}$ is defined, and when $Y=\Spec \Fq$, $C_{n}(X)=C_{n,\Fq}(X):=C_{\Fq^n, \Phi}$ is the natural extension of the commuting variety to the non-affine case, and it parametrizes finite-length coherent sheaves $M$ on $X$ together with an ordered $\Fq$-basis of $H^0(X;M)$ (see \cite[Appendix B]{huangjiang2023punctual}). However, even when $X=\P^1_{\Fq}$, $C_n(X)$ does not seem to have a concrete description. The natural upgrade of \ref{item:commuting-var} is the statement that $C_{N,\Phi}=\Hom_{\mathrm{AssoAlg}_{\O_Y}}(\O_X,{\mathcal{E}nd}_{\O_Y}(N))$, which holds when $\pi$ is affine but not in general.
    \end{enumerate}
\end{example}

\subsection{Framing and stability}\label{subsec:framing}
A fruitful idea to compute the absolute Cohen--Lenstra series is to consider a framed\footnote{Framing and stability are taken from terminologies in quiver varieties and have analogous meanings, see \cite{ginzburg2012lecture}.} analogue and taking the limit as the framing rank approaches infinity. We make it precise in the setting of Example \ref{eg:commuting}\ref{item:commuting-over-ring}; it should take no extra effort to generalize it further to the quiver setting in Example \ref{eg:commuting}\ref{item:commuting-quiver}. See Remark \ref{rmk:framing-appl} for its existing applications.

Fix a commutative ring $S$, and a finitely presented associative $S$-algebra
\begin{equation}
    R=\allowbreak S\{T_1,\dots,T_m\}/(f_1,\dots,f_r).
\end{equation}
For a finite-cardinality $S$-module $N$, let
\begin{equation}
    C_{N,S}(R)=\set{(A_1,\dots,A_m)\in \End_S(N)^m: f_j(A_1,\dots,A_m)=0\text{ for }1\leq j\leq r}
\end{equation}
and
\begin{equation}\label{eq:needed2}
    \Coh_{N,S}(R)=C_{N,S}(R)/\Aut_S(N)=\set{\text{groupoid of $R$-modules with underlying $S$-module $N$}}.
\end{equation}
For $\underline{A}\in C_{N,S}(R)$, we use $(N,\underline{A})$ to denote the $R$-module structure on $N$ determined by $\underline{A}$.

Now for $d\geq 0$, consider
\begin{equation}
    C_{N,S}(R)\times N^d = \set{(\underline{A},v_1,\dots,v_d): \underline{A}\in C_{N,S}(R), v_1,\dots,v_d\in N}.
\end{equation}
To each $(v_1,\dots,v_d)$ above we associate a homomorphism of left $R$-modules $v:R^d\to (N,\underline{A})$ (called \textbf{framing}) given by $v(e_i)=v_i$, where $e_i$ is the standard basis vector of $R^d$. Then we equivalently have
\begin{equation}
    C_{N,S}(R)\times N^d = \set{(\underline{A},v): \underline{A}\in C_{N,S}(R), v\in \Hom_R(R^d,(N,\underline{A}))}.
\end{equation}
We say a framing $v$ is \textbf{stable} if $v$ is surjective, and define
\begin{equation}
    [C_{N,S}(R)\times N^d]_{\st} = \set{(\underline{A},v): \underline{A}\in C_{N,S}(R), v\in \Surj_R(R^d,(N,\underline{A}))}.
\end{equation}
Let $\Aut_S(N)$ act on the set $C_{N,S}(R)\times N^d$ by
\begin{equation}
    g\cdot ((A_1,\dots,A_m),v_1,\dots,v_d) := ((gA_1g^{-1},\dots,gA_mg^{-1}),gv_1,\dots,gv_d).
\end{equation}
Finally, we define the set\footnote{The following construction corresponds to the geometric notion of Quot scheme, hence the notation; see \cite{huangjiang2023torsionfree}.}
\begin{equation}
    \Quot_{R^d,N,S}=\Quot_{d,N,S}(R):=\set{E\subeq_R R^d\text{ is a left $R$-submodule}: R^d/E\sim_S N}.
\end{equation}
The following summarizes the basic connections between these objects.

\begin{lemma}~\label{lem:framing}
    \begin{enumerate}
        \item The above action induces a free action of $\Aut_S(N)$ on $[C_{N,S}(R)\times N^d]_\st$.
        \item The map $[C_{N,S}(R)\times N^d]_{\st}\to \Quot_{R^d,N,S}$ given by $(\underline{A},v)\mapsto \ker(v)$ induces a bijection
        \begin{equation}
            [C_{N,S}(R)\times N^d]_\st/\Aut_S(N) \map[\sim] \Quot_{R^d,N,S}.
        \end{equation}
        \item $\abs{\Quot_{R^d,N,S}}=\dfrac{\abs*{[C_{N,S}(R)\times N^d]_\st}}{\abs{\Aut_S(N)}}.$
    \end{enumerate}
\end{lemma}
\begin{proof}
    The freeness of the action is crucially due to the surjectivity condition on $[C_{N,S}(R)\times N^d]_\st$: Suppose $(\underline{A},v)\in [C_{N,S}(R)\times N^d]_\st$, $g\in \Aut_S(N)$, and $g\cdot (\underline{A},v)=(\underline{A},v)$. We note that $g w(\underline{A})v=w(g\underline{A}g^{-1})gv$ for any word $w$ in $A_1,\dots,A_m$. Combined with our assumption on $g$, we get $ g w(\underline{A})v = w(\underline{A})v$. By the stability assumption, elements of the form $w(\underline{A})v$ generate $N$ as an $S$-module. This implies $g$ is identity on $N$, proving the freeness of the action.
    
    The proof of the remaining assertions is left to the readers.
\end{proof}

Our next goal is to bound the proportion of framings that are not stable. 

\begin{lemma}\label{lem:sur-prob-local}
    If $S$ is a commutative local ring with maximal ideal $\m$ and residue field $\Fq$, and $N$ is a finite $S$-module. Let $r=\dim_{\Fq} N/\m N$. Then the probability that uniformly random $v_1,\dots,v_d\in N$ generate $N$ as an $S$-module is given by $(q^{-(d-r+1)};q^{-1})_r$ if $d\geq r$, and $0$ otherwise.
\end{lemma}
\begin{remark}
    We do not have to assume $S$ is Noetherian. Since $N$ is finite, Nakayama's lemma still applies.
\end{remark}
\begin{proof}
    By Nakayama's lemma, $v_1,\dots,v_d$ generate $N$ if and only if their images generate $N/\m M$, so the desired probability is the probability that an $\Fq$-linear map $\Fq^d\to\Fq^r$ be surjective. 
\end{proof}

\begin{lemma}\label{lem:sur-prob-bound}
    Let $S$ be a commutative ring and $N$ is a finite $S$-module. Then
    \begin{equation}\label{eq:sur-prob-bound}
        \Prob_{v_1,\dots,v_d\in N}(Sv_1+\dots+Sv_d\neq N) \leq 2\abs{N}\log\abs{N} \cdot 2^{-d}.
    \end{equation}
\end{lemma}
\begin{proof}
    Let $V$ be the support of $N$ (the set of maximal ideals such that the localization $N_\m$ is nonzero), so that $N=\bigoplus_{\m \in V} N_\m$. For $\m\in V$, write $q_\m=\abs{S/\m}$ and $r_\m=\dim_{\F_{q_\m}} N_\m / \m N_\m$. Note that $r_\m\geq 1$ by Nakayama's lemma, and $q_\m^{r_\m} \leq \abs{N_\m}$. By Lemma \ref{lem:sur-prob-local} and the fact that $v_1,\dots,v_d$ generate $N$ if and only if for every $\m$ their images generate $N_\m$, we have
    \begin{equation}
        \Prob_{v_1,\dots,v_d\in N}(Sv_1+\dots+Sv_d= N) = \prod_{\m\in V} (q_\m^{-(d-r_\m+1)};q_\m^{-1})_{r_\m}.
    \end{equation}

    By the elementary inequality $\epsilon \leq -\log(1-\epsilon) \leq (2\log 2)\epsilon$ for $0\leq \epsilon\leq 1/2$, we have for $d\geq \max_\m \set{r_\m}$: 
    \begin{align}
        &1-\Prob_{v_1,\dots,v_d\in N}(Sv_1+\dots+Sv_d= N) \\
        &\leq -\log \prod_{\m\in V} (q_\m^{-(d-r_\m+1)};q_\m^{-1})_{r_\m} = \sum_\m \prod_{i=1}^{r_\m}-\log(1-q_\m^{-(d-r_\m+i)}) \\
        &\leq \sum_\m \sum_{i=1}^{r_\m} (2\log 2) q_\m^{-(d-r_\m+i)} \\
        &\leq 2\log 2\sum_{\m} \frac{q_\m^{-1}}{1-q_\m^{-1}}q_\m^{r_\m}q_\m^{-d}.
    \end{align}

    Since $\abs{N}=\prod_{\m\in V} q_\m^{r_\m} \geq \prod_{\m\in V} 2^1$, we have $\abs{V}\leq \log\abs{N}/\log 2$. Using the bounds $q_\m\geq 2$, $\frac{q_\m^{-1}}{1-q_\m^{-1}}\leq 1$, and $q_\m^{r_\m} \leq \abs{N}$, we get
    \begin{equation}
        2\log 2\sum_{\m} \frac{q_\m^{-1}}{1-q_\m^{-1}}q_\m^{r_\m}q_\m^{-d} \leq (2\log 2) \frac{\log \abs{N}}{\log 2} \abs{N} 2^{-d},
    \end{equation}
    which proves the desired inequality. Finally, we may remove the assumption $d\geq \max_\m \set{r_\m}$: noting that $\abs{N}=\prod_{\m\in V} q_\m^{r_\m} \geq 2^{\max_{\m}\set{r_\m}}$, so $2^d\geq N$ implies $d\geq \max_\m \set{r_\m}$; on the other hand, if $N>2^d$ (so $N\geq 2$), the inequality \eqref{eq:sur-prob-bound} is vacuous since the right-hand side is at least $2\log 2>1$. 
\end{proof}

\begin{corollary}
    In the notation above, 
    \begin{equation}\label{eq:inequality-needed}
        \frac{\abs*{C_{N,S}(R)\times N^d}-\abs*{[C_{N,S}(R)\times N^d]_\st}}{\abs*{C_{N,S}(R)\times N^d}} \leq 2\abs{N}\log\abs{N} \cdot 2^{-d}.
    \end{equation}
\end{corollary}
\begin{proof}
    If $v_1,\dots,v_d\in N$ do not generate $(N,\underline{A})$ as an $R$-module, then $v_1,\dots,v_d$ do not generate $N$ as an $S$-module. Therefore, $(C_{N,S}(R)\times N^d) \setminus [C_{N,S}(R)\times N^d]_\st$ is contained in the locus where $Sv_1+\dots+Sv_d\neq N$. By Lemma~\ref{lem:sur-prob-bound}, the conditional probability that $(\underline{A},v)\in [C_{N,S}(R)\times N^d]_\st$ for fixed $\underline{A}\notin C_{N,S}(R)$ and uniformly random $v_1,\dots,v_d\in N$ is bounded above by $2\abs{N}\log\abs{N} \cdot 2^{-d}$. Hence so is the unconditional probability.
\end{proof}

We now state and prove an effective ``framing rank $\to\infty$'' formula, which is the goal of the section.

\begin{proposition}
    Let $S$ be a commutative ring and $R$ be a finitely presented associative $S$-algebra. Then for any finite-cardinality $S$-module $N$, we have 
    \begin{equation}\label{eq:limit-formula-effective}
        0\leq \abs{\Coh_{N,S}(R)} - \frac{\abs{\Quot_{R^d,N,S}}}{\abs{N}^d} \leq 2\abs{N}\log\abs{N}\abs{\Coh_{N,S}(R)} \cdot 2^{-d}.
    \end{equation}
    In particular,
    \begin{equation}\label{eq:limit-formula}
        \abs{\Coh_{N,S}(R)} = \lim_{d\to \infty} \frac{\abs{\Quot_{R^d,N,S}}}{\abs{N}^d}.
    \end{equation}
\end{proposition}
\begin{proof}
    The desired inequality follows from multiplying the both sides of \eqref{eq:inequality-needed} by $\frac{\abs*{C_{N,S}(R)\times N^d}}{\abs{N}^d \abs{\Aut_S(N)}}$, and substituting \eqref{eq:needed2} and Lemma \ref{lem:framing}(c).
\end{proof}

\begin{remark}
    By the same proof, if the residue field cardinality of every $\m$ in the support of $N$ is at least $q$ (for example, when $S$ is an algebra over $\Fq$, or $\Zp$ with $q=p$), then we have a better bound
    \begin{equation}
        0\leq \abs{\Coh_{N,S}(R)} - \frac{\abs{\Quot_{R^d,N,S}}}{\abs{N}^d} \leq \abs{N}\log\abs{N}\abs{\Coh_{N,S}(R)} \cdot q^{1-d}.
    \end{equation}
\end{remark}

If $S$ is a finitely generated commutative ring over $\Z$, or a completion thereof, then \eqref{eq:limit-formula} has a zeta-function interpretation. Define a formal Dirichlet series (called the \textbf{Quot zeta function} of rank $d$ over $R$)
\begin{equation}
    \zeta_{R^d}(s):=\sum_{E\subeq_R R^d} \abs{R^d/E}^{-s},
\end{equation}
where the sum ranges over left $R$-submodules of finite index in $R^d$. Grouping all $E$ in terms of the $S$-module structure of $R^d/E$ gives
\begin{equation}
    \zeta_{R^d}(s)=\sum_{N\in \Ob(\FinMod_S)/{\sim}} \abs{\Quot_{R^d,N,S}} \abs{N}^{-s}.
\end{equation}
Also recall the Cohen--Lenstra zeta function $\zetahat_R(s)$ defined in Example \ref{eg:absolute-cl}\ref{item:absolute-cl-zeta}.

\begin{proposition}\label{prop:limit-formula-zeta}
    Assume the notation above, and let $a_n(\cdot)$ denote the $n^{-s}$ coefficient of a formal Dirichlet series. Then for $d\geq 0$,
    \begin{equation}
        0\leq a_n(\zetahat_R(s))-a_n(\zeta_{R^d}(s+d)) \leq 2n\log n \cdot a_n(\zetahat_R(s))\cdot 2^{-d}.
    \end{equation}
    As a consequence, we have a coefficientwise limit of formal Dirichlet series
    \begin{equation}
        \zetahat_R(s)=\lim_{d\to \infty} \zeta_{R^d}(s+d).
    \end{equation}
\end{proposition}
\begin{proof}
    For $n\geq 1$, let $\mathcal{S}_n$ denote the set $\set{N\in \Ob(\FinMod_S)/{\sim}: \abs{N}=n}$. Then we have $a_n(\zetahat_R(s))=\sum_{N\in \mathcal{S}_n} \abs{\Coh_{N,S}(R)}$ and $a_n(\zeta_{R^d}(s+d))=\sum_{N\in \mathcal{S}_n} \abs{\Quot_{R^d,N,S}} \abs{N}^{-d}$. By \eqref{eq:limit-formula-effective},
    \begin{align}
        0\leq a_n(\zetahat_R(s))-a_n(\zeta_{R^d}(s+d)) &= \sum_{N\in \mathcal{S}_n} \parens*{\abs{\Coh_{N,S}(R)}-\abs{\Quot_{R^d,N,S}} \abs{N}^{-d}} \\
        &\leq 2n\log n\sum_{N\in \mathcal{S}_n} \abs{\Coh_{N,S}(R)} 2^{-d}\\
        &=2n\log n \cdot a_n(\zetahat_R(s))\cdot 2^{-d}. \qedhere
    \end{align}
\end{proof}

\begin{remark}\label{rmk:framing-appl}
    The idea of approximating $\zetahat_R(s)$ by framed analogues has appeared implicitly in \cite{cohenlenstra1984heuristics} when they considered the ``$k$-weight'', see for example \cite[Proposition 3.1]{cohenlenstra1984heuristics}. In \cite{huangjiang2023torsionfree}, this idea was crucially used to compute the Cohen--Lenstra zeta functions for some singular curves. 
\end{remark}

\section{Applications}\label{sec:applications}
\subsection{Polynomial ring in one variable over a Dedekind domain}
Let $S$ be a Dedekind domain such that the Dedekind zeta function $\zeta_S(s)$ is defined (see \S \ref{subsec:dedekind}), and let $R=S[T]$ be the polynomial ring in one variable. Our first goal is to compute the Cohen--Lenstra zeta function $\zetahat_R(s)$. 

\begin{lemma}\label{lem:feit-fine-dvr}
    If $(S,\pi,\Fq)$ is a DVR, then
    \begin{equation}
        \zetahat_{S[T]}(s)=\prod_{i,j\geq 1} \frac{1}{1- q^{1-is-j}}.
    \end{equation}
\end{lemma}
\begin{proof}
    By \eqref{eq:cl-zeta-in-comm} applied to $R=S[T]$, we have
    \begin{equation}
        \zetahat_{S[T]}(s)=\sum_{N\in \Ob(\FinMod_S)/{\sim}} \frac{\abs{\End_S(N)}}{\abs{\Aut_S(N)}} \abs{N}^{-s}.
    \end{equation}
    Since finite modules over $S$ are classified by partitions (see \S \ref{subsec:partition}), by Lemma \ref{lem:aut-and-end}, we get
    \begin{align}
        \zetahat_{S[T]}(s)&= \sum_{\lambda} \frac{q^{\sum_{i\geq 1}\lambda_i'^2} }{q^{\sum_{i\geq 1}\lambda_i'^2}\prod_{i\geq 1}(q^{-1};q^{-1})_{m_i(\lambda)}} q^{-s\abs{\lambda}} \\
        &=\sum_{m_1,m_2,\dots \geq 0} \frac{1}{(q^{-1};q^{-1})_{m_i}} q^{-s\sum_{i\geq 1} im_i} \\
        &=\prod_{i\geq 1} \sum_{m=0}^\infty \frac{(q^{-is})^m}{\prod_{i\geq 1}(q^{-1};q^{-1})_m} \\
        &=\prod_{i\geq 1} \frac{1}{(q^{-is};q^{-1})_\infty}=\prod_{i,j\geq 1} \frac{1}{1-q^{-is} q^{1-j}},
    \end{align}
    as required.
\end{proof}

\begin{proposition}\label{prop:feit-fine}
    If $S$ is a Dedekind domain with a well-defined Dedekind zeta function, then
    \begin{equation}
        \zetahat_{S[T]}(s)=\prod_{i,j\geq 1} \zeta_S(is+j-1).
    \end{equation}
\end{proposition}
\begin{proof}
    Let $\mathcal{R}=\FinMod_{S[T]}$ and $\mathcal{R_\m}=\FinMod_{S_\m[T]}$ for a maximal ideal $\m$ of $S$. Denote $q_\m :=\abs{S/\m}$. Use the measure $\mu(M)=\abs{M}^{-s}$. Then the hypothesis of Lemma \ref{lem:euler-prod} is verified, giving $\zetahat_{S[T]}(s)=\prod_\m \zetahat_{S_\m[T]}(s)$. By Lemma \ref{lem:feit-fine-dvr} applied to $S_\m$, we get
    \begin{equation}
        \zetahat_{S[T]}(s)= \prod_\m \prod_{i,j\geq 1} \frac{1}{1- q_\m^{1-is-j}}=\prod_{i,j\geq 1} \zeta_S(is+j-1)
    \end{equation}
    by the Euler product \eqref{eq:dedekind-euler-prof}, proving the claimed formula.
\end{proof}

We recover the famous formula of Feit--Fine.
\begin{corollary}[Feit--Fine \cite{feitfine1960pairs}]
    We have
    \begin{equation}
        \sum_{n\geq 0} \frac{\#\set{(A,B)\in \Mat_n(\Fq):AB=BA}}{\abs{\GL_n(\Fq)}} t^n = \prod_{i,j\geq 1} \frac{1}{1-t^i q^{2-j}}.
    \end{equation}
\end{corollary}
\begin{proof}
    Set $t=q^{-s}$. The left-hand side can be recognized as $\zetahat_{\Fq[x,y]}(s)$ (for instance, by \eqref{eq:cl-zeta-in-comm} with $R=\Fq[x,y]$ and $S=\Fq$). To obtain the right-hand side, we apply Proposition \ref{prop:feit-fine} with $S=\Fq[x]$, and substitute
    \begin{equation}\label{eq:line-zeta}
        \zeta_{\Fq[x]}(s)=\frac{1}{1-q^{1-s}}. \qedhere
    \end{equation}
\end{proof}

One can easily recover relevant formulas and asymptotics from our Dirichlet series. For example, let $a_n$ be the number of modules $M$ over $\Z[T]$ with cardinality $n$ up to isomorphism, counted with weight $1/\abs{\Aut_{\Z[T]}(M)}$. Then $\zetahat_{\Z[T]}(s)=\sum_{n\geq 1} a_n n^{-s}$. From Proposition \ref{prop:feit-fine} with $S=\Z$, $\zetahat_{\Z[T]}(s)$ is holomorphic on $\Re(s)>1/2$ except a pole at $s=1$ with residue $\prod_{j\geq 2} \zeta(j)^j$. By a Tauberian theorem (for instance \cite[Lemma 5.2]{cohenlenstra1984heuristics}), we conclude that
\begin{equation}
    \sum_{n=1}^N a_n\sim \parens*{\prod_{j\geq 2} \zeta(j)^j } \log N\text{ as }N\to \infty.
\end{equation}
As for the exact formula, we have
\begin{equation}
    a_n=\sum_{(n_{ij})_{i,j\geq 1}, \prod_{i,j} n_{ij}^i=n} n_{ij}^{1-j}.
\end{equation}

Using Proposition \ref{prop:limit-formula-zeta}, we can bound a related quantity. Let $b_{d,n}$ be the number of $\Z[T]$-submodules of $\Z[T]^d$ of index $n$. Then $\zeta_{R^d}(s)=\sum_{n\geq 1} b_{d,n}n^{-s}$. By Proposition \ref{prop:limit-formula-zeta} and noting that the $n$-th Dirichlet coefficient if $\zeta_{R^d}(s+d)$ is $n^{-d}b_{d,n}$, we get
\begin{equation}
    b_{d,n}=\left(n^d+O(2n\log n\cdot (n/2)^d)\right)a_n,
\end{equation}
where the implied constant in big-$O$ is absolute (in fact, $1$ suffices). In view of the work \cite{moschettiricolfi2018} on $\Fq[x,y]$ and the analogy between $\Fq[x]$ and $\Z$, the exact computation of $b_{d,n}$ is expected to be difficult.

\subsection{Some nonreduced curves}
One motivation of this subsection is to compute the Cohen--Lenstra series of some possibly nonreduced curves in a plane, namely, $x^b=0$ and $x^b y=0$. These amount to counting solutions to the systems of matrix equations $\set{AB=BA, A^b=0}$ and $\set{AB=BA, A^b B=0}$, respectively. However, we will prove more general formulas that also apply to situations without a commuting matrix interpretation. 

\subsubsection{}
Fix $b\geq 1$, and let $(S,\pi,\Fq)$ be a discrete valuation ring. We first compute $\zetahat_{S[T]/(\pi^b)}$.

\begin{lemma}\label{lem:fat-line}
    In the setting above, we have 
    \begin{equation}
        \zetahat_{S[T]/(\pi^b)}=\prod_{i=1}^b\prod_{j\geq 0} \frac{1}{1-q^{-is-j}}.
    \end{equation}
\end{lemma}
\begin{proof}
    By viewing $S[T]/(\pi^b)=(S/\pi^b)[T]$ as an algebra over $S/\pi^b$, \eqref{eq:cl-zeta-in-comm} reads
    \begin{equation}
        \zetahat_{S[T]/(\pi^b)}(s)=\sum_{N\in \Ob(\FinMod_{S/\pi^b})/{\sim}} \frac{\abs{\End_S(N)}}{\abs{\Aut_S(N)}} \abs{N}^{-s}.
    \end{equation}
    Since finite modules over $S/\pi^b$ are finite modules over $S$ whose types have parts bounded above by $b$, the same argument as Lemma \ref{lem:feit-fine-dvr} gives
    \begin{align}
        \zetahat_{S[T]/(\pi^b)}(s)&= \sum_{\lambda: \lambda_1\leq b} \frac{q^{\sum_{i\geq 1}\lambda_i'^2}}{q^{\sum_{i\geq 1}\lambda_i'^2}\prod_{i\geq 1} (q^{-1};q^{-1})_{m_i(\lambda)}} q^{-s\abs{\lambda}} \\
        &=\sum_{m_1,m_2,\dots,m_b \geq 0} \frac{1}{\prod_{i\geq 1}(q^{-1};q^{-1})_{m_i}} q^{-s\sum_{i\geq 1} im_i} \\
        &=\prod_{i=1}^b\prod_{j\geq 0} \frac{1}{1-q^{-is-j}},
    \end{align}
    as desired.
\end{proof}

\begin{corollary}
    We have
    \begin{equation}
        \sum_{n\geq 0} \frac{\#\set{(A,B)\in \Mat_n(\Fq):AB=BA, A^b=0}}{\abs{\GL_n(\Fq)}} t^n=\prod_{i=1}^b \prod_{j\geq 0} \frac{1}{1-t^i q^{-j}}.
    \end{equation}
\end{corollary}
\begin{proof}
    The left-hand side is $\zetahat_{\Fq[x,y]/x^b}(s)$ with $t=q^{-s}$. Then apply Lemma \ref{lem:fat-line} with $S=\Fq[x]$ and $\pi=x$. 
\end{proof}

\subsubsection{}
Next, we compute $\zetahat_{S[T]/(\pi^b T)}$ for $b\geq 1$ and a DVR $(S,\pi,\Fq)$. We first study an auxiliary series that will appear in the computation.

\begin{definition}
    Define the power series
    \begin{equation}
        \Zhat(t,u,q):=\sum_{\lambda} \frac{q^{\sum_{i\geq 1}\lambda_i'^2}}{\prod_{i\geq 1} (q;q)_{m_i(\lambda)}} t^{\abs{\lambda}} u^{\ell(\lambda)}\in \Z[[t,u,q]].
    \end{equation}
\end{definition}
    
Note that if we let $S$ be as above, set $t=q^{-s}$ as usual, and consider the measure $\nu(N)=\abs{N}^{-s} u^{\dim_{\Fq} N/\pi N}\in \C[[t,u]]$ for $N\in \FinMod_S$, then 
\begin{equation}
    \Zhat(t,u,q^{-1})=\zetahat_{\FinMod_S,\nu}.
\end{equation}
When $u=1$, by recognizing that $\zetahat_{\FinMod_S,\nu}|_{u=1}=\zetahat_S(s)$, we have by \cite{cohenlenstra1984heuristics}
\begin{equation}\label{eq:needed4}
    \Zhat(t,1,q)=\frac{1}{(tq;q)_\infty}.
\end{equation}
Other specializations of $u$, such as $u=t^b$ where $b\geq 1$, do not seem to have a product form. However, we have the following formula. (For experts: it is a basic hypergeometric series $\Hyp{0}{1}{-}{tq}{q,tuq}$.)

\begin{lemma}
    We have
    \begin{equation}\label{eq:cl-with-rank}
        \Zhat(t,u,q)=\sum_{k=0}^\infty \frac{q^{k^2} t^k u^k}{(q;q)_k (tq;q)_k}.
    \end{equation}
    Moreover, $(tq;q)_\infty \Zhat(t,u,q)$ converges for $t,u\in \C$ and $\abs{q}<1$.
\end{lemma}
\begin{proof}
    Suppose the Durfee partition of a partition $\Lambda$ is $\sigma(\Lambda)=\lambda'$ (for terminology and notation, see \S \ref{subsec:durfee}), then $\Lambda$ can be built from all its Durfee squares, the subpartition to the right of the first Durfee square, the subpartition to the right of the second Durfee squares and below the first Durfee square, and so on. Combined with \eqref{eq:identity-durfee} and $m_i(\lambda)=\lambda_i'-\lambda_{i+1}'$, we thus have
    \begin{equation}
        \frac{q^{\sum_{i\geq 1}\lambda_i'^2}}{\prod_{i\geq 1} (q;q)_{m_i(\lambda)}} = \sum_{\Lambda: \sigma(\Lambda)=\lambda'} q^{\Lambda}
    \end{equation}

    Summing over all $\lambda$, and setting $\lambda'=\sigma(\Lambda)$ and $k=\ell(\lambda)=\lambda_1'=\sigma_1(\Lambda)$, we get
    \begin{align}
        \Zhat(t,u,q)&=\sum_{\Lambda} q^{\Lambda} t^{\ell(\Lambda)} u^{\sigma_1(\Lambda)} \\
        &=\sum_{k=0}^\infty u^k \sum_{\Lambda:\sigma_1(\Lambda)=k} q^{\Lambda} t^{\ell(\Lambda)}.
    \end{align}

    Given $k$, a partition $\Lambda$ with $\sigma_1(\Lambda)=k$ is determined by a partition $\Lambda^{(1)}$ with $\ell(\Lambda^{(1)})\leq k$ and a partition $\Lambda^{(2)}$ with $\Lambda^{(2)}_1\leq k$. It follows that
    \begin{align}
        \sum_{\Lambda:\sigma_1(\Lambda)=k} q^{\Lambda} t^{\ell(\Lambda)} &= \sum_{\Lambda^{(1)},\Lambda^{(2)}} q^{k^2+\abs{\Lambda^{(1)}}+\abs{\Lambda^{(2)}}} t^{k+\ell(\Lambda^{(2)})}
        \\&= q^{k^2} t^k \sum_{\Lambda^{(1)}} q^{\abs{\Lambda^{(1)}}} \sum_{\Lambda^{(2)}} q^{\abs{\Lambda^{(2)}}}t^{\ell(\Lambda^{(2)})} = \frac{q^{k^2}t^k}{(q;q)_k (tq;q)_k},
    \end{align}
    where the last equality is by \eqref{eq:identity-durfee} again. This finishes the proof of \eqref{eq:cl-with-rank}.

    Finally, when $t,u\in \C$ and $\abs{q}<1$, the convergence of
    \begin{equation}
        (tq;q)_\infty \Zhat(t,u,q) = \sum_{k=0}^\infty \frac{q^{k^2} t^k u^k}{(q;q)_k}  (tq^{k+1};q)_\infty
    \end{equation}
    results from the rapidly decaying factor $q^{k^2}$, see the argument of \cite[Proposition 5.1(a)]{huang2023mutually}.
\end{proof}

\begin{proposition}\label{prop:nonred-node}
    Let $b\geq 1$ and $(S,\pi,\Fq)$ be a DVR. Set $t=q^{-s}$. Then
    \begin{equation}
        \zetahat_{S[T]/(\pi^b T)}(s)=\parens*{\prod_{i=1}^b\prod_{j\geq 0} \frac{1}{1-t^i q^{-j}} } \sum_{k=0}^\infty \frac{q^{-k^2} t^{b+1}}{(q^{-1};q^{-1})_k (tq^{-1};q^{-1})_k}.
    \end{equation}
\end{proposition}

\begin{proof}
    By viewing $S[T]/(\pi^b T)$ as an algebra over $S$, \eqref{eq:cl-zeta-in-comm} reads
    \begin{equation}
        \zetahat_{S[T]/(\pi^b T)}(s)=\sum_{N\in \Ob(\FinMod_{S})/{\sim}} \frac{\#\set{A\in \End_S(N): \pi^b A=0}}{\abs{\Aut_S(N)}} \abs{N}^{-s}.
    \end{equation}
    Note that $\set{A\in \End_S(N): \pi^b A=0}$ is the $\pi^b$ torsion $\End_S(N)[\pi^b]$ of $\End_S(N)$, and taking the $\pi^b$ torsion of a finite $S$-module amounts to keeping only the first $b$ columns of its partition. Recall also that if the type of $N$ is $\lambda$, then the type of $\End_S(N)$ is $\lambda^2$ (see the notation in \eqref{subsec:partition}). Thus,
    \begin{equation}
        \abs*{\End_S(N)[\pi^b]}=q^{\sum_{i=1}^b \lambda_i'^2}.
    \end{equation}

    It follows that if we set $t=q^{-s}$, then
    \begin{align}
        \zetahat_{S[T]/(\pi^b T)}(s)&= \sum_{\lambda} \frac{q^{\sum_{i=1}^b\lambda_i'^2} }{q^{\sum_{i\geq 1}\lambda_i'^2}\prod_{i\geq 1}(q^{-1};q^{-1})_{m_i(\lambda)}} t^{\abs{\lambda}} \\
        &=\sum_{\lambda} \frac{q^{-\sum_{i\geq b+1} \lambda_i'^2}}{(q^{-1};q^{-1})_{m_i(\lambda)}} t^{\abs{\lambda}}.
    \end{align}

    We now factorize this series. Note that $\lambda_i'=m_i(\lambda)+m_{i+1}(\lambda)+\dots$, so $\sum_{i\geq b+1} \lambda_i'^2$ depends only on $m_i(\lambda)$ with $i\geq b+1$. As a result,
    \begin{align}
        \zetahat_{S[T]/(\pi^b T)}(s)&=\parens*{\prod_{i=1}^b \frac{1}{(q^{-1};q^{-1})_{m_i}} t^{im_i}} \sum_{m_{b+1},m_{b+2}\dots\geq 0} \frac{q^{-\sum_{i\geq b+1} \lambda_i'^2}}{\prod_{i\geq b+1}(q^{-1};q^{-1})_{m_i}} t^{\sum_{i\geq b+1}im_i}.
    \end{align}

    For each $\lambda$, consider a partition $\rho$ given by $m_i(\rho)=m_{b+i}(\lambda)$; equivalently, $\rho$ is obtained from $\lambda$ by removing the first $b$ columns. Hence,
    \begin{equation}
        \sum_{m_{b+1},m_{b+2}\dots\geq 0} \frac{q^{-\sum_{i\geq b+1} \lambda_i'^2}}{\prod_{i\geq b+1}(q^{-1};q^{-1})_{m_i}} t^{\sum_{i\geq b+1}im_i} = \sum_\rho \frac{q^{-\sum_{i\geq 1} \rho_i'^2}}{\prod_{i\geq 1} (q^{-1};q^{-1})_{m_i(\rho)}} t^{\abs{\rho}+b\cdot \ell(\rho)} = \Zhat(t,t^b,q^{-1}).
    \end{equation}

    By \eqref{eq:euler-identity} and \eqref{eq:cl-with-rank}, we finally get
    \begin{align}
        \zetahat_{S[T]/(\pi^b T)}(s) &= \frac{1}{\prod_{i=1}^b (t^i q^{-1};q^{-1})_\infty} \Zhat(t,t^b,q^{-1})\\
        &= \parens*{\prod_{i=1}^b\prod_{j\geq 0} \frac{1}{1-t^i q^{-j}} } \sum_{k=0}^\infty \frac{q^{-k^2} t^{b+1}}{(q^{-1};q^{-1})_k (tq^{-1};q^{-1})_k},
    \end{align}
    finishing the proof.
\end{proof}

\subsubsection{} We now prove a global analogue. Define
\begin{equation}
    H(t,u,q):=(tq;q)_\infty \Zhat(t,u,q)=\sum_{k=0}^\infty \frac{q^{k^2} t^k u^k}{(q;q)_k}  (tq^{k+1};q)_\infty\in \Z[t,u,q],
\end{equation}
and we recall that $H(t,u,q)$ converges when $t,u\in \C$ and $\abs{q}<1$. Note also that $H(t,1,q)=1$ from \eqref{eq:needed4}. We need the following restatement of Proposition~\ref{prop:nonred-node}.

\begin{lemma}
    If $(S,\pi,\Fq)$ is a DVR, then
    \begin{equation}
        \frac{\zetahat_{S[T]/(\pi^b T)}}{\zetahat_{S}(s) \zetahat_{S[T]/(\pi^b)}(s)} = H(q^{-s},q^{-bs},q^{-1}).
    \end{equation}
\end{lemma}
\begin{proof}
    This follows directly from Proposition \ref{prop:nonred-node}, Lemma \ref{lem:fat-line}, and the well-known $\zetahat_S(s)=(q^{-1-s};q^{-1})_\infty^{-1}$. 
\end{proof}

We are ready to move on to the global setting. Let $S$ be a Dedekind domain such that the Dedekind zeta function is defined. Let $\fa$ be a nonzero ideal of $S$, so $\fa$ has a unique factorization $\fa=\prod_\m \m^{b_\m}$ into product of maximal ideals, where all but finitely many $b_\m$ are zero. Let $q_\m=\abs{S/\m}$ and $V=V(\fa)=\set{\m:b_\m\neq 0}$. 

\begin{proposition}\label{prop:nonred-node-global}
    In the notation above,
    \begin{equation}
        \zetahat_{S[T]/(\fa T)}(s)=\parens*{\prod_{i=1}^\infty \zeta_S(s+i)} \prod_{\m\in V(\fa)} \parens*{H(q_\m^{-s},q_\m^{-b_\m s},q_\m^{-1}) \prod_{i=1}^{b_\m}\prod_{j\geq 0} \frac{1}{1-q_\m^{-is-j}}}.
    \end{equation}
\end{proposition}
\begin{proof}
    Using Lemma \ref{lem:euler-prod} in a way analogous to the proof of Proposition \ref{prop:feit-fine}, we get
    \begin{equation}\label{eq:nonred-node-global}
        \frac{\zetahat_{S[T]/(\fa T)}(s)}{\zetahat_{S}(s) \zetahat_{S[T]/\fa}(s)} = \prod_{\m\in V} \frac{\zetahat_{S_\m[T]/(\m^{b_\m} T)}(s)}{\zetahat_{S_\m}(s) \zetahat_{S_\m[T]/\m^{b_\m}}(s)}= \prod_{\m\in V}  H(q_\m^{-s},q_\m^{-b_\m s},q_\m^{-1})
    \end{equation}
    by the preceeding lemma. The claimed formula follows from $\zetahat_{S[T]/\fa}(s)=\prod_\m \zetahat_{S_\m[T]/\m^{b_\m}}(s)$, Lemma \ref{lem:fat-line}, and the famous formula of Cohen and Lenstra \cite[\S 7]{cohenlenstra1984heuristics}: $\zetahat_S(s)=\prod_{i=1}^\infty \zeta_S(s+i)$.
\end{proof}

\begin{corollary}\label{cor:nonred-node}
    Let $b\geq 1$. Setting $t=q^{-s}$, we have
    \begin{align}
        \zetahat_{\Fq[x,y]/(x^b y)}(s)&= \sum_{n\geq 0} \frac{\#\set{(A,B)\in \Mat_n(\Fq):AB=BA, A^b B=0}}{\abs{\GL_n(\Fq)}} t^n\\
        & = \frac{1}{(t;q^{-1})_\infty \prod_{i=1}^b (t^i;q^{-1})_\infty}\sum_{k=0}^\infty \frac{q^{-k^2} t^{(b+1)k}}{(q^{-1};q^{-1})_k}  (tq^{-(k+1)};q^{-1})_\infty. \label{eq:nonred-node}
    \end{align}
\end{corollary}
\begin{proof}
    Let $S=\Fq[x]$ and $\fa=(x^b)$, then the only nonzero $b_\m$ is $b_{\m}=1$ for $\m=(x)$. The rest follows from Proposition \ref{prop:nonred-node-global} and Equation \eqref{eq:line-zeta}.
\end{proof}

\begin{remark}
    The case $b=1$ recovers a result that has three distinct proofs so far: two from matrix counting (\cite{huang2023mutually}, \cite{fulmanguralnick2022}) and one using the framing technique in \S \ref{subsec:framing} (\cite{huangjiang2023torsionfree}). Even in this case, our proof is different from all those three.
\end{remark}

\begin{remark}\label{rmk:resolve}
    A cleaner restatement of \eqref{eq:nonred-node} is probably
    \begin{equation}
        \frac{\zetahat_{\Fq[x,y]/(x^b y)}(s)}{\zetahat_{\Fq[x,y]/(y)}(s)\zetahat_{\Fq[x,y]/(x^b)}(s)}=H(t,t^b,q^{-1}),
    \end{equation}
    obtained directly from \eqref{eq:nonred-node-global}. We note that $x^b y=0$ is a union of the line $y=0$ and the thickened line $x^b=0$. The convergence of $H(t,t^b,q^{-1})$ for all $t$ provides a nonreduced evidence to the author's conjecture \cite{huang2023mutually} that quotients of Cohen--Lenstra series arising from ``desingularization'' should be entire. 
\end{remark}

\subsection{Towards commuting triples}\label{subsec:triples}
The main idea of \S \ref{sec:general} naturally leads to the observation that counting commuting triples of matrices over $\Fq$ is determined by counting conjugacy classes (or irreducible representations over $\C$) of the finite group $G_\lambda(q):=\Aut_{\Fq[[x]]}(N_\lambda)$ for every partition $\lambda$, where $N_\lambda$ is an $\Fq[[x]]$-module of type $\lambda$. We briefly explain why. By the orbit-stabilizer theorem, the count of conjugacy classes of $G_\lambda(q)$ determines the number of commuting pairs in $G_\lambda(q)$. This is $\abs{C_{N_\lambda,\Fq[[x]]}(\Fq[[x]][y^{\pm 1},z^{\pm 1}])}$.\footnote{The trick is to rewrite the equation $AB=BA, A,B\in G_{\lambda}(q)$ as $A,A',B,B'\in \End_{\Fq[[x]]}(N_\lambda)$ pairwise commute with $AA'=BB'=\mathrm{id}$. Then note that $\Fq[[x]][y,y',z,z']/(yy'-1,zz'-1)=\Fq[[x]][y^{\pm 1},z^{\pm 1}]$.} By \eqref{eq:cl-zeta-in-comm}, knowing this for all $\lambda$ determines $\zetahat_{\Fq[[x]][y^{\pm 1},z^{\pm 1}]}(s)$. By a general machinery in algebraic geometry called the power structure (see \cite{bryanmorrison2015motivic}), this and $\zetahat_{\Fq[x,y,z]}(s)$ determine each other.

While counting conjugacy classes in $G_\lambda(q)$ is probably an equally hard problem in general, it has been approached by techniques that seem remote from the standard toolkits in the research on commuting varieties. It is even conjectured \cite{onn2008representations} that this count is a polynomial in $q$ for every $\lambda$, as is mentioned in the introduction \S \ref{sec:intro}. If true, Onn's conjecture would imply that $\abs{C_{n,3}(\Fq)}$ is a polynomial in $q$, so that extracting its leading term would give the dimension of $C_{n,3}(\C)$. 

\section*{Acknowledgements} The author thanks Asvin G, Ruofan Jiang, and Yifan Wei for stimulating discussions.


\begin{thebibliography}{HLRV13}

    \bibitem[And98]{andrewspartitions}
    George~E. Andrews.
    \newblock {\em The theory of partitions}.
    \newblock Cambridge Mathematical Library. Cambridge University Press,
      Cambridge, 1998.
    \newblock Reprint of the 1976 original.
    
    \bibitem[BF98]{bryanfulman1998orbifold}
    Jim Bryan and Jason Fulman.
    \newblock Orbifold {E}uler characteristics and the number of commuting
      {$m$}-tuples in the symmetric groups.
    \newblock {\em Ann. Comb.}, 2(1):1--6, 1998.
    
    \bibitem[BM15]{bryanmorrison2015motivic}
    Jim Bryan and Andrew Morrison.
    \newblock Motivic classes of commuting varieties via power structures.
    \newblock {\em J. Algebraic Geom.}, 24(1):183--199, 2015.
    
    \bibitem[CL84]{cohenlenstra1984heuristics}
    H.~Cohen and H.~W. Lenstra, Jr.
    \newblock Heuristics on class groups of number fields.
    \newblock In {\em Number theory, {N}oordwijkerhout 1983 ({N}oordwijkerhout,
      1983)}, volume 1068 of {\em Lecture Notes in Math.}, pages 33--62. Springer,
      Berlin, 1984.
    
    \bibitem[Dro72]{drozd1972}
    Ju.~A. Drozd.
    \newblock Representations of commutative algebras.
    \newblock {\em Funkcional. Anal. i Prilo\v{z}en.}, 6(4):41--43, 1972.
    
    \bibitem[Dro80]{drozd1980tame}
    Ju.~A. Drozd.
    \newblock Tame and wild matrix problems.
    \newblock In {\em Representation Theory II}, volume 832, pages 242--258.
      Springer, Berlin, Heidelberg, 1980.
    
    \bibitem[FF60]{feitfine1960pairs}
    Walter Feit and N.~J. Fine.
    \newblock Pairs of commuting matrices over a finite field.
    \newblock {\em Duke Math. J.}, 27:91--94, 1960.
    
    \bibitem[FG22]{fulmanguralnick2022}
    Jason Fulman and Robert Guralnick.
    \newblock Cohen {L}enstra partitions and mutually annihilating matrices over a
      finite field.
    \newblock {\em Linear Algebra Appl.}, 645:1--8, 2022.
    
    \bibitem[Ger61]{gerstenhaber1961}
    Murray Gerstenhaber.
    \newblock On dominance and varieties of commuting matrices.
    \newblock {\em Ann. of Math. (2)}, 73:324--348, 1961.
    
    \bibitem[Gin12]{ginzburg2012lecture}
    Victor Ginzburg.
    \newblock Lectures on {N}akajima's quiver varieties.
    \newblock In {\em Geometric methods in representation theory. {I}}, volume 24-I
      of {\em S\'{e}min. Congr.}, pages 145--219. Soc. Math. France, Paris, 2012.
    
    \bibitem[GS00]{guralnicksethuraman2000}
    Robert~M. Guralnick and B.~A. Sethuraman.
    \newblock Commuting pairs and triples of matrices and related varieties.
    \newblock {\em Linear Algebra Appl.}, 310(1-3):139--148, 2000.
    
    \bibitem[Gur92]{guralnick1992note}
    Robert~M. Guralnick.
    \newblock A note on commuting pairs of matrices.
    \newblock {\em Linear and Multilinear Algebra}, 31(1-4):71--75, 1992.
    
    \bibitem[HJ23a]{huangjiang2023torsionfree}
    Yifeng Huang and Ruofan Jiang.
    \newblock Generating series for torsion-free bundles over singular curves:
      rationality, duality and modularity.
    \newblock Preprint \url{https://arxiv.org/abs/2312.12528}, 2023.
    
    \bibitem[HJ23b]{huangjiang2023punctual}
    Yifeng Huang and Ruofan Jiang.
    \newblock Punctual {Q}uot schemes and {C}ohen--{L}enstra series of the cusp
      singularity.
    \newblock Preprint \url{https://arxiv.org/abs/2305.06411}, 2023.
    
    \bibitem[HLRV13]{hlrv2013positivity}
    Tam\'{a}s Hausel, Emmanuel Letellier, and Fernando Rodriguez-Villegas.
    \newblock Positivity for {K}ac polynomials and {DT}-invariants of quivers.
    \newblock {\em Ann. of Math. (2)}, 177(3):1147--1168, 2013.
    
    \bibitem[HO01]{holbrookomladic2001}
    John Holbrook and Matja\v{z} Omladi\v{c}.
    \newblock Approximating commuting operators.
    \newblock {\em Linear Algebra Appl.}, 327(1-3):131--149, 2001.
    
    \bibitem[HO15]{holbrookomeara2015}
    J.~Holbrook and K.~C. O'Meara.
    \newblock Some thoughts on {G}erstenhaber's theorem.
    \newblock {\em Linear Algebra Appl.}, 466:267--295, 2015.
    
    \bibitem[HOS23]{hos2023matrix}
    Yifeng Huang, Ken Ono, and Hasan Saad.
    \newblock Counting matrix points on certain varieties over finite fields.
    \newblock {\em Contemp. Math., Amer. Math. Soc.}, accepted for publication,
      2023.
    \newblock \url{https://arxiv.org/abs/2302.04830}.
    
    \bibitem[Hua23]{huang2023mutually}
    Yifeng Huang.
    \newblock Mutually annihilating matrices, and a {C}ohen--{L}enstra series for
      the nodal singularity.
    \newblock {\em J. Algebra}, 619:26--50, 2023.
    
    \bibitem[J{\v{S}}22]{jelisiejewsivic}
    Joachim Jelisiejew and Klemen {\v{S}}ivic.
    \newblock Components and singularities of {Q}uot schemes and varieties of
      commuting matrices.
    \newblock {\em J. Reine Angew. Math.}, 788:129--187, 2022.
    
    \bibitem[Kac83]{kac1983root}
    Victor~G. Kac.
    \newblock Root systems, representations of quivers and invariant theory.
    \newblock In {\em Invariant theory ({M}ontecatini, 1982)}, volume 996 of {\em
      Lecture Notes in Math.}, pages 74--108. Springer, Berlin, 1983.
    
    \bibitem[Kin94]{king1994moduli}
    A.~D. King.
    \newblock Moduli of representations of finite-dimensional algebras.
    \newblock {\em Quart. J. Math. Oxford Ser. (2)}, 45(180):515--530, 1994.
    
    \bibitem[Mac15]{macdonaldsymmetric}
    I.~G. Macdonald.
    \newblock {\em Symmetric functions and {H}all polynomials}.
    \newblock Oxford Classic Texts in the Physical Sciences. The Clarendon Press,
      Oxford University Press, New York, second edition, 2015.
    
    \bibitem[MR18]{moschettiricolfi2018}
    Riccardo Moschetti and Andrea~T. Ricolfi.
    \newblock On coherent sheaves of small length on the affine plane.
    \newblock {\em J. Algebra}, 516:471--489, 2018.
    
    \bibitem[MT55]{motzkintaussky1955}
    T.~S. Motzkin and Olga Taussky.
    \newblock Pairs of matrices with property {$L$}. {II}.
    \newblock {\em Trans. Amer. Math. Soc.}, 80:387--401, 1955.
    
    \bibitem[N{\v{S}}14]{ngosivic2014varieties}
    Nham~V. Ngo and Klemen {\v{S}}ivic.
    \newblock On varieties of commuting nilpotent matrices.
    \newblock {\em Linear Algebra Appl.}, 452:237--262, 2014.
    
    \bibitem[Onn08]{onn2008representations}
    Uri Onn.
    \newblock Representations of automorphism groups of finite
      {$\mathfrak{o}$}-modules of rank two.
    \newblock {\em Adv. Math.}, 219(6):2058--2085, 2008.
    
    \bibitem[PSS15]{pss2015similarity}
    Amritanshu Prasad, Pooja Singla, and Steven Spallone.
    \newblock Similarity of matrices over local rings of length two.
    \newblock {\em Indiana Univ. Math. J.}, 64(2):471--514, 2015.
    
    \bibitem[{\v{S}}iv12]{sivic2012iii}
    Klemen {\v{S}}ivic.
    \newblock On varieties of commuting triples {III}.
    \newblock {\em Linear Algebra Appl.}, 437(2):393--460, 2012.
    
    \bibitem[Whi13]{white2013counting}
    Tad White.
    \newblock Counting free {A}belian actions.
    \newblock Preprint. \url{https://arxiv.org/pdf/1304.2830.pdf}, 2013.
    
    \end{thebibliography}
\end{document}